\documentclass[12pt, a4paper]{article}

\usepackage{pdfsync}
\usepackage{amsmath}\multlinegap=0pt
\usepackage[latin1]{inputenc}
\usepackage{amsthm, amsmath}
\usepackage{amssymb}
\usepackage{mathtools,thmtools,stmaryrd,esint,mathrsfs}
\usepackage{calrsfs}
\usepackage{hyperref}
\usepackage{graphicx}
\usepackage{xcolor}
\numberwithin{equation}{section}
\theoremstyle{plain}
\newtheorem{thm}{Theorem}[section]

\newtheorem{prop}[thm]{Proposition}
\newtheorem{lem}[thm]{Lemma}

\theoremstyle{definition}

\theoremstyle{remark}
\newtheorem{rem}[thm]{Remark}

\newcommand{\R}{\mathbb{R}}

\newcommand\N{{\mathbb N}}

\newcommand\pref[1]{(\ref{#1})}

\let \eps\varepsilon

\DeclareMathOperator{\spt}{spt}

\DeclareMathOperator{\id}{id}

\newcommand\dist{\mathop{\mathrm{dist}}\nolimits}

\newcommand\dive{\mathrm{div}}

\newcommand\Lip{\mathrm{Lip}}

\def\<#1,#2>{\left<#1,#2\right>}

\newcommand\tal{{\widetilde {\alpha}}}
\newcommand\tgam{{\widetilde {\gamma}}}
\newcommand\tmu{{\widetilde {\mu}}}

\newcommand{\Prob}{\mathscr{P}}
\newcommand{\dd}{\mathrm{d}}
\newcommand{\OT}{\mathrm{OT}}

\newcommand\vv{\boldsymbol{v}}
\newcommand\Lom{\mathscr{L}_\Omega}

\def\al{\alpha}
\def\be{\beta}
\def\ga{\gamma}
\def\ta{\tau}

\usepackage{ifthen}

\newboolean{colorText}
\setboolean{colorText}{true} % change to'false' to desactivate the coloration

\newcommand{\chg}[1]{%
\ifthenelse{\boolean{colorText}}{\textcolor{blue}{#1}}{#1}%
}

\title{Well-posedness and convergence of entropic approximation of semi-geostrophic equations}

\author{Guillaume Carlier\thanks{CEREMADE, UMR CNRS 7534, Universit\'e Paris
Dauphine, PSL, Pl. de Lattre de Tassigny, 75775 Paris Cedex 16, FRANCE and INRIA-Paris, MOKAPLAN,
\texttt{carlier@ceremade.dauphine.fr}},
Hugo Malamut\thanks{CEREMADE, UMR CNRS 7534, Universit\'e Paris
Dauphine, PSL, Pl. de Lattre de Tassigny, 75775 Paris Cedex 16, FRANCE and INRIA-Paris, MOKAPLAN,
\texttt{hugo.malamut@ens.psl.eu}}
}

\begin{document}

\maketitle

\begin{abstract}
We prove existence and uniqueness of solutions for an entropic version of the semi-geostrophic equations. We also establish convergence to a weak solution of the semi-geostrophic equations as the entropic parameter vanishes. Convergence is also proved for discretizations that can be computed numerically in practice as shown recently in \cite{BCM23}.
\end{abstract}


\smallskip

\textbf{MS Classification:} 49Q22, 35D30, 35Q35.

 \section{Introduction}

 The semi-geostrophic equations are used in meteorology to describe frontogenesis for large scale atmospheric flows. Initially proposed by Eliassen \cite{Eliassen} in the late 1940s and subsequently developed by Hoskins \cite{Hoskins} and Cullen, see \cite{Cullenbook0}, \cite{Cullenbook}, \cite{CullenPurser},  they have gained a lot of attention by mathematicians because of their connections with optimal transport theory. The seminal work of Brenier \cite{Brenier}  indeed enabled Benamou and Brenier \cite{BenamouBrenier98} to define a notion of weak solutions and  establish the existence of such solutions. This  was later generalized by Loeper \cite{Loeper} in the framework of measure-valued solutions (also see Feldman and Tudorascu \cite{FT1, FT2}).  It is also worth mentioning here that Ambrosio and Gangbo \cite{AmbrosioGangbo} developed a general and self-contained theory of Hamiltonian systems in the Wasserstein space which contains semi-geostrophic equations as a special case. Regarding stronger notions of solutions, thanks to the deep regularity theory for the Monge-Amp\`ere equation, Ambrosio,  Colombo, De Philippis and  Figalli were able to obtain Eulerian solutions in the two-dimensional periodic case \cite{FigalliColombo2012} and for convex three-dimensional domains \cite{FigalliColombo2014}. 
 
 \smallskip

 Given $\Omega$ a Lipschitz bounded open subset of $\R^3$,  and $\alpha_0$ a Borel measure on $\R^3$ with total mass $\vert \Omega\vert$ the semi-geostrophic system reads as the coupling of the continuity equation
 \begin{equation}\label{sg1}
 \partial_t \alpha + \dive(\alpha J (\id-\nabla \psi))=0,  \; \alpha(0, \cdot)=\alpha_0, \mbox{ with } J:= \left(\begin{array}{ccc }
0 & -1 & 0 \\
1 & 0 & 0   \\
0 & 0 & 0
\end{array}\right)
\end{equation}
with the Monge-Amp\`ere equation
 \begin{equation}\label{sg2}
 \det(D^2 \psi_t) = \alpha_t, \; \psi_t \mbox{ convex.}
 \end{equation}
 Before going further, let us briefly explain the origin of \eqref{sg1}-\eqref{sg2},  we refer to \cite{Cullenbook0}, \cite{Cullenbook}  \cite{CullenPurser} and the references therein for a modern and complete exposition and a derivation from incompressible Euler equations with a large Coriolis force. The starting point is the following system
 \[\left\{
\begin{array}{rr}
D_t \vv^g + \nabla p =-J \vv+ \theta e_3, \; \dive (\vv)=0, \;  \mbox{ in } (0,\infty)\times \Omega\\
 \vv^g = J \nabla p,  \; D_t \theta =0, \;  \mbox{ in } (0,\infty)\times \Omega,\\
 \vv \cdot \nu_{\Omega}=0,  \mbox{ on } (0,\infty)\times \partial\Omega,\;  p(0, \cdot)=p_0,
\end{array}
\right.\]
where $D_t=\partial_t + \vv \cdot \nabla$ denotes the convective derivative, $e_3=(0,0,1)^T$ is the vertical direction and $\nu_{\Omega}$ denotes the outward unit normal to $\partial \Omega$. Note that the first equation enables one to relate the pressure $p$ to the \emph{buoyancy} $\theta$ through the simple relation $\partial_3 p=\theta$ and that the so-called semi-geostrophic balance equation  $ \vv^g = J \nabla p$ imposes that the \emph{semi-geostrophic velocity field} $\vv^g$ is horizontal. Cullen's stability criterion \cite{CullenPurser} further  requires that
\[P_t(x)=p_t(x)+\frac{1}{2}(x_1^2+ x_2^2)\]
is convex. Substituting  $\vv^g= J \nabla p$ in the first equation of the system yields
\[D_t \nabla P= J \nabla p=J(\nabla P-\id).\]
Let us then denote by $(t,x)\mapsto X_t(x)$ the flow of $\vv$.  This flow remains  in the physical domain $\Omega$  (since $\vv$ is tangential to $\partial \Omega$) and is  measure preserving, i.e. ${X_t}_\# \Lom=\Lom$ where $\Lom$ denotes the Lebesgue measure on $\Omega$ (since $\dive(\vv)=0$).  At least formally, we have
\[\begin{split}
\partial_t \nabla P_t(X_t(x))&=J(\nabla P_t(X_t(x))-X_t(x))\\
&= J(\id -\nabla P_t^*) ( \nabla P_t(X_t(x)))\end{split}\]
where $P_t^*$ is the Legendre transform of $P_t$.  So if we set $\alpha_t:={\nabla P_t}_\# \Lom$ and $\psi_t:= P_t^*$ (so that $\nabla \psi_t=(\nabla P_t)^{-1}$ and  $\Lom={\nabla \psi_t}_\# \alpha_t$ if $\nabla P_t$ is invertible and smooth), the semi-geostrophic system reads as the coupling of the continuity equation \eqref{sg1} (with initial condition $\alpha_0={\nabla P_0}_\# \Lom$) with the Monge-Amp\`ere equation \eqref{sg2}. Note that \eqref{sg1}-\eqref{sg2} is a reformulation of the initial semi-geostrophic system in \emph{semi-geostrophic} coordinates i.e. after the change of variables $x\mapsto \nabla P_t(x)$. Conversely, note that (again formally) from a solution of \eqref{sg1}-\eqref{sg2} one recovers $\nabla P=\nabla \psi^*$, $\theta$ by $\theta=\partial_3 p=\partial_3 P$, $\vv^g= J \nabla p$, and the velocity $\vv$ by
\[\vv_t(x)=\partial_t \nabla \psi_t (\nabla P_t(x)) + (D^2 \psi_t( \nabla P_t(x))) J(\nabla P_t(x)-x).\] 
In general,  $\psi_t$ and $P_t$ may fail to be smooth or strongly convex  so that the previous computations only remain  formal and one has to resort to some notion of weak solution, following the seminal optimal transport approach of  Benamou-Brenier \cite{BenamouBrenier98} (which will be recalled in section \ref{sec-entreg}) based on the fact that $\nabla P_t$ is the optimal transport from $\Lom$ to $\alpha_t$. 

\smallskip 
 
 The present paper is motivated by recent research on optimal transport methods to numerically solve semi-geostrophic equations.  On the one hand, in the spirit of the geometric approach pioneered by Cullen and Purser \cite{CullenPurser}, Bourne, Egan, Pelloni and Wilkinson  proposed  a semi-discrete optimal transport strategy \cite{Bourne22} that was implemented in \cite{BourneJCP}. On the other hand, Benamou, Cotter and Malamut \cite{BCM23} developed and implemented an alternative method based on entropic approximation of optimal transport and the popular Sinkhorn algorithm (see \cite{Cuturi}, \cite{CuturiPeyre}). In \cite{Bourne22}, the authors have established well-posedness of semi-discrete approximations (for suitably discretized initial conditions) as well as their convergence. Concerning entropic approximation of semi-geostrophic equations,  well-posedness and convergence issues are not addressed in \cite{BCM23}. The purpose of our paper is precisely to study these two questions.  The paper is organized as follows. In section \ref{sec-entreg}, we recall fundamentals from optimal transport theory, the weak formulation of semi-geostrophic equations and introduce their entropic approximations with regularization parameter $\eps>0$. We establish well-posedness of these approximations  for every $\eps>0$ in section \ref{sec-wp}. Section \ref{sec-convergence} is devoted to convergence as $\eps \to 0^+$, including the convergence analysis of the explicit discretization scheme recently proposed  in \cite{BCM23}. 
 
 \section{Optimal and entropic optimal transport for semi-geostrophic equations}\label{sec-entreg}

{\textbf{Notations}} We shall denote by $\Prob(\R^d)$ the set of Borel probability measures on $\R^d$, by $\Prob_c(\R^d)$ the set of compactly supported  Borel probability measures on $\R^d$, for $R>0$, $B_R$ denotes the closed ball of radius $R$ of $\R^d$, centered at $0$ and $\Prob(B_R)$ the set of Borel probability measures supported on $B_R$. For $x\in \R^d$, we denote by $\vert x \vert$ the euclidean norm of $x$ and if  $A$ is a $d\times d$ matrix, $\Vert A \Vert$ stands for the euclidean operator norm of $A$. 
 
\smallskip
 
 {\textbf{Optimal transport (OT) and weak solutions of the semi-geostrophic system}} Given $R>0$, $\alpha$ and $\mu$ in $\Prob(B_R)$, denoting by $\Pi(\alpha, \mu)$ the set of transport plans between $\alpha$ and $\mu$ i.e. the set of  Borel probability measures on $\R^d\times \R^d$ having $\alpha$ and $\mu$ as marginals, the squared-2-Wasserstein distance between $\alpha$ and $\mu$ is defined as the value of the quadratic OT problem
 \begin{equation}\label{W2}
 W_2^2(\alpha, \mu):=\inf_{\gamma \in \Pi(\alpha, \mu)}  \int_{\R^d\times \R^d} \vert x-y \vert^2 \gamma (\dd x, \dd y).
 \end{equation}
It is a classical result of Brenier \cite{Brenier} that, as soon as $\mu$ is absolutely continuous with respect to the Lebesgue measure, \pref{W2} admits a unique optimal plan $\gamma$ which is characterized by $\gamma=(\nabla \varphi, \id)_\# \mu$ i.e. 
\[ \int_{\R^d\times \R^d} f(x, y) \gamma (\dd x, \dd y)= \int_{\R^d} f(\nabla \varphi(y), y) \mu (\dd y), \; \forall f\in C_c(\R^d\times \R^d)\]
with $\varphi$ convex (and $\nabla \varphi$ unique up to $\mu$-negligible sets). Note in particular that $\nabla \varphi_\# \mu=\alpha$ and defining $\psi$ as the Legendre transform of $\varphi$:
\[\psi(x):=\sup_{y \in B_R} \{x \cdot y-\varphi(y)\}, \; \forall x\in \R^d\]
the optimal plan $\gamma$ is supported by pairs $(x,y)$ for which $y\in \partial \psi(x)$ (where $\partial \psi(x)$ denotes the subdifferential of the convex function $\psi$ at $x$, see \cite{Rock}). In particular, if we disintegrate $\gamma$ with respect to its first marginal $\alpha$:
\[ \gamma(\dd x, \dd y ) =\alpha(\dd x) \gamma(\dd y \vert x)\]
so that $\gamma(. \vert x)$ is a probability measure on $B_{R}$, its conditional expectation
\[\int_{\R^d} y \gamma(\dd y \vert x)\]
belongs to $\partial \psi(x)$ by convexity of $\partial \psi(x)$.
 Note also that $\psi$ can be seen, at least formally, as a weak solution of the Monge-Amp\`ere equation
\[\det(D^2 \psi) \mu(\nabla \psi)=\alpha.\]
Let us now fix $R_0>0$, $\mu_0 \in \Prob(B_{R_0})$, $\alpha_0$ in $\Prob(B_{R_0})$; let $A$ be a $d\times d$ matrix and $T>0$ be a time-horizon. We are interested in the following evolution problem for a curve of probability measures $\alpha$: $t\in [0, T] \mapsto \alpha_t \in \Prob_c(\R^d)$:
\begin{equation}\label{sg3}
 \partial_t \alpha + \dive(\alpha A (\id-\nabla \psi))=0,  \; \alpha(0,\cdot)=\alpha_0, \det(D^2 \psi_t) \mu_0(\nabla \psi_t) = \alpha_t, \; \psi_t \mbox{ convex.}
\end{equation}
Of course, for $d=3$, $\mu_0$ the uniform probability measure on $\Omega$ and $A= J$, we recover the initial problem \pref{sg1}-\pref{sg2} after normalizing all measures by dividing them by $\vert \Omega\vert$. In view of the previous discussion, it is natural to define weak solutions of \pref{sg3} as follows. A weak solution of \pref{sg3} on $[0,T]\times \R^d$, is a continuous for $W_2$ curve of probability measures $t\in [0,T] \mapsto \alpha_t \in \Prob_c(\R^d)$ such that:
\begin{itemize}
\item  $\spt(\alpha_t) \subset B_{ 2 R_0 e^{\Vert A \Vert T}}$ for every $t\in [0,T]$ (which by  Gr\"onwall's lemma, is consistent with the linear  growth of the velocity field: $\vert A(x-\nabla \psi(x)) \vert \leq \Vert A \Vert (\vert x \vert+ R_0)$ and the fact that the initial condition $\alpha_0$ is supported in $B_{R_0}$),

\item for every $f \in C^1([0,T]\times \R^d)$, one has 
\begin{multline}\label{sgweak1}
\int_0^T \int_{\R^d} [\partial_t f + Ax \cdot \nabla f]\alpha_t(\dd x) \dd t-\int_0^T \int_{\R^d\times \R^d} A y \cdot \nabla f(t,x) \gamma_t(\dd x, \dd y) \dd t\\
=  \int_{\R^d}  f(T,x) \alpha_T (\dd x)- \int_{\R^d} f(0, x) \alpha_0(\dd x),
\end{multline}
where $\gamma_t$ is an optimal plan between $\alpha_t$ and $\mu_0$ i.e. $\gamma_t \in \Pi(\alpha_t, \mu_0)$ and 
\begin{equation}\label{sgweak2}
W_2^2(\alpha_t, \mu_0) =\int_{\R^d \times \R^d} \vert x-y \vert^2 \gamma_t(\dd x, \dd y), \mbox{ 
for a.e. $t\in [0,T].$} 
\end{equation}

\end{itemize}
Note that we have not imposed that $\mu_0$ is absolutely continuous with respect to the Lebesgue measure here (so that there may not exist optimal transport maps between $\mu_0$ and $\alpha_t$ which is the reason the weak formulation above has been written in terms of-possibly non unique-optimal plans).

 \smallskip
 
{\textbf{Entropic OT and entropic semi-geostrophic equations}} Given $\eps>0$, $\alpha$ and $\mu$ in $\Prob_c(\R^d)$, the entropic regularization of the quadratic OT problem \pref{W2} reads 
\begin{equation}\label{ote}
\OT_\eps(\alpha, \mu):=\inf_{\gamma \in \Pi(\alpha, \mu)} \Big\{ \frac{1}{2} \int_{\R^d\times \R^d} \vert x-y \vert^2 \gamma (\dd x, \dd y) + \eps H(\gamma \vert \alpha \otimes \mu)\Big\}
\end{equation} 
 where $H(p\vert q)$ denotes the relative entropy of $p$ with respect to $q$:
 \[H(p\vert q)=\begin{cases} \int_{\R^d} \log\Big( \frac{\dd p }{ \dd q}(x)\Big) p(\dd x) \mbox{ if $p \ll q$,} \\ + \infty \mbox{ otherwise.}\end{cases}\]
 There exists a unique optimal plan $\gamma^\eps$ for $\OT_\eps(\alpha, \mu)$ and it is well-known that $\gamma^\eps$ has the Gibbs form
 \begin{equation}
 \gamma^{\eps}(\dd x , \dd y)=\exp\Big( - \frac{\vert x-y\vert^2}{2 \eps} + \frac{u^\eps(y)+ v^\eps(x)}{\eps}  \Big) \alpha(\dd x )  \mu (\dd y)
 \end{equation}
 where the  potentials $u^\eps$ and $v^\eps$ are such that $\gamma^\eps \in \Pi(\alpha, \mu)$ i.e. satisfy the so-called Schr\"odinger system
  \begin{equation}\label{schro1}
  v^\eps(x)=-\eps \log \Big( \int_{\R^d} e^{-\frac{\vert x-y\vert^2}{2 \eps} +  \frac{u^\eps(y) }{\eps} }  \mu (\dd y)  \Big)  \mbox{ for $\alpha$-a.e. $x$,}
  \end{equation}
and
  \begin{equation}\label{schro2}
  u^\eps(y)=-\eps \log \Big( \int_{\R^d} e^{-\frac{\vert x-y\vert^2}{2 \eps} +  \frac{v^\eps(x) }{\eps} }  \alpha (\dd x)  \Big)  \mbox{ for $\mu$-a.e. $y$.}
  \end{equation}
  These potentials are called Schr\"odinger potentials for the entropic optimal transport problem between $\alpha$ and $\mu$. For an overview of entropic optimal transport, we refer the reader to \cite{Leonard14},  the lecture notes \cite{Nutz22} and the rerences therein. 
 Note that formulas \pref{schro1} and \pref{schro2} can be extended to define smooth maps on the whole of $\R^d$ (which are defined uniquely up to additive constants) and that disintegrating the optimal entropic plan $\gamma^\eps$
 as 
 \[ \gamma^\eps(\dd x, \dd y ) =\alpha(\dd x) \gamma^\eps(\dd y \vert x)\]
 one has for every $x\in \R^d$
 \begin{equation}\label{gradve}
 \nabla v^\eps(x)=\frac{  \int_{\R^d} (x-y) e^{-\frac{\vert x-y\vert^2}{2 \eps} +  \frac{u^\eps(y) }{\eps} }  \mu (\dd y)  }  {   \int_{\R^d} e^{-\frac{\vert x-y\vert^2}{2 \eps} +  \frac{u^\eps(y) }{\eps} }  \mu (\dd y)  } = x-\int_{\R^d} y \gamma^\eps(\dd y \vert x)
 \end{equation}
 and
  \begin{equation}\label{hessve}
 D^2 v^\eps (x)= \id-\eps^{-1} \Big(\int_{\R^d} y\otimes y \gamma^\eps(\dd y \vert x)  -\int_{\R^d} y \gamma^\eps(\dd y \vert x) \otimes \int_{\R^d} y \gamma^\eps(\dd y \vert x) \Big). 
 \end{equation}

Before defining the entropic regularization of the semi-geostrophic system, let us recall  two results that will be useful in the following. The first result follows from  Proposition 3.1 of \cite{CPT} which gives an error estimate of the form $\OT_\eps- \frac{1}{2} W_2^2=O(\eps \vert \log(\eps)\vert)$  or the (more precise) results from \cite{MalamutSylvestre} where  an $O(\eps)$ suboptimality estimate for entropic plans is established. A non quantitative consequence of these results can be stated as:
 \begin{prop}
  \label{prop:uniform}
      Let $R>0$ and for $\alpha,\beta \in \Prob(B_R)^2$, denote by $\gamma^\eps_{\al,\be}$ the solution of the entropic optimal transport problem between $\alpha$ and $\beta$. Then
       $\int_{\R^d\times\R^d} |x-y|^2 \; \ga_{\al,\be}^\eps(\dd x, \dd y)$ converges to $W_2^2(\alpha,\beta)$ uniformly in $\alpha,\beta \in \Prob(B_R)^2$ as $\eps \to 0$.  
  \end{prop}
  The second result concerns the stability of the Schr\"odinger potentials with respect to the marginals\footnote{It is worth mentioning   that the constant $\Delta(\eps, R)$ in \eqref{displacementsmooth} derived  from the analysis of \cite{CCL} scales exponentially badly with $\eps$. Interestingly, in the very recent paper \cite{divol2024tight}  close (but not obviously applicable to our context) results are obtained  with much better constants  of the order of $R^2 \eps^{-1}$.} and follows from Theorem 1.1 of \cite{CCL}:
  \begin{prop}
  \label{prop:stability}
  Let $R>0$ and $\mu_0 \in \Prob(B_R)$. For $\al \in \Prob(B_R)$, let $v^\eps[\al]$ be the Schr\"odinger potential for the entropic optimal transport problem between $\al$ and $\mu_0$. Then, there is a  constant $\Delta(\eps, R)$ depending on $\eps$ and $R$ such that for every $\al_1,\al_2 \in \Prob(B_R)^2$, we have 
  \begin{equation}\label{displacementsmooth}
   \|\nabla v^\eps[\al_1] - \nabla v^\eps[\al_2]\|_{L^\infty(B_R)} \leq \Delta(\eps, R) W_2(\al_1,\al_2).
   \end{equation}
  \end{prop} 
  
 Given  $T>0$, $R_0>0$, $\mu_0$ and $\alpha_0$  in $\Prob(B_{R_0})$, we now consider the following entropic regularization with parameter $\eps>0$ of \pref{sg3} 
  \begin{equation}\label{sgeps1}
 \partial_t \alpha^\eps + \dive(\alpha^\eps A (\nabla v^\eps))=0 \mbox{ in $[0,T]\times \R^d$},  \; \alpha^\eps(0,\cdot )=\alpha_0, 
 \end{equation}
 where 
 \begin{equation}\label{sgeps2}
 \nabla v^\eps_t(x)=  x-\int_{\R^d} y \gamma^\eps_t(\dd y \vert x)
 \end{equation}
 and $\gamma^\eps_t$ is the solution of $\OT^\eps(\alpha_t^\eps, \mu_0)$:
 \begin{equation}\label{sgeps3}
 \gamma^\eps_t \in \Pi(\alpha_t^\eps, \mu_0), \;  \OT_\eps(\alpha_t^\eps, \mu_0)=\frac{1}{2} \int_{\R^d\times \R^d} \vert x-y \vert^2 \gamma_t^\eps (\dd x, \dd y) + \eps H(\gamma^\eps_t \vert \alpha^\eps_t \otimes \mu_0).
 \end{equation}
 Since $\gamma^\eps_t$ has second marginal $\mu_0$ supported on $B_{R_0}$, the conditional probability $\gamma^\eps_t(. \vert x)$ is supported on $B_{R_0}$, as well so that $\nabla v^\eps$ has controlled linear growth
 \begin{equation}
 \vert \nabla v^\eps_t(x) \vert \leq \vert x \vert + R_0, \; \forall x\in \R^d, \; \forall t.
 \end{equation}
 Therefore one defines weak solutions of \pref{sgeps1}-\pref{sgeps2}-\pref{sgeps3} in a similar way as we did for \pref{sg3}. A weak solution of  \pref{sgeps1}-\pref{sgeps2}-\pref{sgeps3} on $[0,T]\times \R^d$, is a continuous for $W_2$ curve of probability measures $t\in [0,T] \mapsto \alpha_t^\eps \in \Prob_c(\R^d)$ such that:
\begin{itemize}
\item  $\spt(\alpha_t^\eps) \subset B_{2 R_0 e^{\Vert A \Vert T}}$ for every $t\in [0,T]$, 
\item for every $f \in C^1([0,T]\times \R^d)$, one has 
\begin{multline}\label{sgweakeps1}
\int_0^T \int_{\R^d} [\partial_t f + Ax \cdot \nabla f]\alpha_t^\eps(\dd x) \dd t-\int_0^T \int_{\R^d\times \R^d} A y \cdot \nabla f(t,x) \gamma_t^\eps(\dd x, \dd y) \dd t\\
=  \int_{\R^d}  f(T,x) \alpha_T^\eps (\dd x)- \int_{\R^d} f(0, x) \alpha_0(\dd x),
\end{multline}
where  for a.e. $t\in [0,T]$, $\gamma_t^\eps$ is the solution of $\OT_\eps(\alpha_t^\eps, \mu_0)$ i.e. \pref{sgeps3} holds. 
 \end{itemize}

 \section{Well-posedness for fixed $\eps>0$}\label{sec-wp}
 
 To prove well-posedness of entropic semi-geostrophic equations, we will show that these equations enjoy nice Lipschitz properties (detailed in paragraph \ref{subsec-fixed} below) for which an elementary fixed-point argument can be used.

 \subsection{A fixed-point argument}\label{subsec-fixed}

 Consider a map $B$ : $\alpha \in \Prob_c(\R^d) \mapsto B[\alpha] \in C(\R^d, \R^d)$ with the following properties:
 
 \begin{itemize}
 
 \item {\bf{(H1)}} There exists $C>0$ such that 
 \begin{equation}\label{growthcond}
 \vert B[\alpha](x)\vert \leq C(1 + \vert x\vert), \; \forall (x, \alpha)\in \R^d \times \Prob_c(\R^d).
 \end{equation}
 
 \item  {\bf{(H2)}} For every $R>0$ 
 \begin{equation}\label{lipxcond}
 K_R:=\sup \{ \Lip(B[\alpha], B_R), \; \alpha \in \Prob(B_R)\}<+\infty.
 \end{equation}
 
 \item {\bf{(H3)}} For every $R>0$ 
 \begin{equation}\label{lipalphacond}
 M_R:=\sup \Big\{ \frac{ \Vert B[\alpha^1] -B[\alpha^2] \Vert_{L^\infty(B_R)}} {W_2(\alpha^1, \alpha^2)}, \;  (\alpha^1, \alpha^2) \in \Prob(B_R)^2, \; \alpha^1 \neq \alpha^2 \Big\}<+\infty.
 \end{equation}

 \end{itemize}

 \begin{lem}\label{lembasic}
 Let $B$ : $\alpha \in \Prob_c(\R^d) \mapsto B[\alpha] \in C(\R^d, \R^d)$ satisfy {\bf{(H1)}-\bf{(H2)}-\bf{(H3)}}, $R_0>0$, $\alpha_0 \in \Prob(B_{R_0})$ and $T>0$. Setting $R_T:=(R_0+1)e^{CT}$, there exists a unique curve $t\in [0,T] \mapsto \alpha_t \in \Prob(B_{R_T})$ which solves 
 \begin{equation}\label{transporteqB}
 \partial_t \alpha + \dive(\alpha B[\alpha] )=0, \;  \alpha(0,\cdot)=\alpha_0, 
 \end{equation}
 in the weak sense on $[0,T]\times \R^d$.
 \end{lem}
 \begin{proof}
 Set $K=K_{R_T}$, $M=M_{R_T}$ (see \pref{lipxcond} and \pref{lipalphacond}) and  
 \begin{equation}\label{choixlambda}
 \lambda:=\frac{5M}{2}+ K.
 \end{equation}
 Define 
 \[E_T:=C([0,T], (\Prob(B_{R_T}), W_2))\]
 which, equipped with the distance 
 \[\dist(\alpha^1, \alpha^2):=\sup_{t\in [0,T]} e^{-\lambda t} W_2(\alpha_t^1, \alpha_t^2), \; (\alpha^1, \alpha_2)=(\alpha_t^1, \alpha_t^2)_{t\in [0,T]} \in E_T^2\]
 is a complete metric space.  By well-known arguments on the continuity equation (see Chapter 8 in \cite{AGS}), our assumptions on $B$ imply that $\alpha \in E_T$ solves \eqref{transporteqB} if and only if
 \[\alpha =\Phi_{\alpha_0}(\alpha) \mbox{ with } \Phi_{\alpha_0}(\alpha)_t:= {X_t^{\alpha}}_\# \alpha_0, \; t\in [0,T]\]
 where $X_t^{\alpha}$ is the (globally well-defined) flow of $B[\alpha]$:
 \begin{equation}
 \frac{\dd} {\dd t} X_t^{\alpha}(x) = B[\alpha_t](X_t^{\alpha}(x)),
 X_0^{\alpha}(x)=x, \; (t,x)\in [0,T]\times \R^d.
 \end{equation}
 By Gr\"onwall's Lemma and \pref{growthcond} $\Phi_{\alpha_0}$ is a self map of $E_T$, we shall now prove that it is a contraction. Let $(\alpha^1, \alpha^2)\in E_T^2$, to shorten notation set $B^i (t,x)=B[\alpha_t^i](x)$, $i=1, 2$ and let $X_t^1$, $X_t^2$ denote the flows of $B^1$ and $B^2$ respectively. Then, we first obviously have
 \begin{equation}\label{ineqwf}
 W_2^2(\Phi_{\alpha_0}(\alpha^1)_t, \Phi_{\alpha_0}(\alpha^2)_t)=W_2^2({X_t^1}_{\#} \alpha_0, {X_t^2}_{\#} \alpha_0) \leq \int_{\R^d} \vert X_t^1(x)-X_t^2(x) \vert^2 \alpha_0(\dd x)
 \end{equation}
 Next we observe that for fixed $x\in B_{R_0}$, 
 \[\begin{split}
 \frac{\dd} {\dd t} \vert X_t^1(x)-X_t^2(x)\vert^2 = 2(X_t^1(x)-X_t^2(x))\cdot (B^1(t,X_t^1(x)) -B^2(t,X_t^2(x)))  \\
 = 2(X_t^1(x)-X_t^2(x))\cdot (B^1(t,X_t^1(x)) -B^1(t,X_t^2(x)))\\
 + 2(X_t^1(x)-X_t^2(x))\cdot (B^1(t,X_t^2(x)) -B^2(t,X_t^2(x)))\\
 \leq 2K \vert X_t^1(x)-X_t^2(x)\vert^2 + 2M W_2(\alpha_t^1, \alpha_t^2) \vert X_t^1(x)-X_t^2(x)\vert \\
 \leq (2K+M) \vert X_t^1(x)-X_t^2(x)\vert^2+ M \dist^2(\alpha^1, \alpha^2) e^{2\lambda t}
 \end{split}\]
 where we have used \pref{lipxcond}-\pref{lipalphacond}  and Cauchy-Schwarz for the first inequality and Young's inequality and the definition of $\dist$ in the last line. Thanks to Gr\"onwall's Lemma, we get by our choice of $\lambda$ in \pref{choixlambda}:
 \[  \vert X_t^1(x)-X_t^2(x)\vert^2 e^{-2\lambda t} \leq \frac{M \dist^2(\alpha^1, \alpha^2) }{2 \lambda - 2K-M}= \frac{\dist^2(\alpha^1, \alpha^2)}{4}.\]
 Together with \pref{ineqwf}, integrating the previous inequality with respect to $\alpha_0$ yields
 \[\dist^2(\Phi_{\alpha_0}(\alpha^1), \Phi_{\alpha_0}(\alpha^2))=\sup_{t\in [0, T]} e^{-2 \lambda t} W_2^2(\Phi_{\alpha_0}(\alpha^1)_t, \Phi_{\alpha_0}(\alpha^2)_t )\leq  \frac{\dist^2(\alpha^1, \alpha^2)}{4}\]
 which shows that $\Phi$ is $\frac{1}{2}$-Lipschitz hence a contraction  for $\dist$. The existence and uniqueness of a fixed-point follows from the Banach-Picard fixed-point Theorem which completes the proof. 
  \end{proof}
 \begin{rem}
 The above proof also straighforwardly implies the Lipschitz dependence of the solution of \eqref{transporteqB} with respect to the initial condition. Indeed, let $\alpha_0$ and $\tal_0$ be in $\Prob(B_{R_0})$ $t\in [0,T]\mapsto \alpha_t$ and $t\in [0,T]\mapsto \tal_t$ denote the solution of  \eqref{transporteqB}  with respective initial conditions $\alpha_0$ and $\tal_0$. Then, defining $K$, $M$, $\lambda$,  $\dist$, $\Phi_{\alpha_0}$ and $\Phi_{\tal_0}$ as in the proof of Lemma \ref{lembasic}, we have thanks to the triangle inequality and the fact that $\Phi_{\alpha_0}$ is $\frac{1}{2}$-Lipschitz
 \[ \begin{split}
 \dist(\alpha, \tal)&= \dist(\Phi_{\alpha_0} (\alpha), \Phi_{\tal_0}(\tal)) \leq  \dist(\Phi_{\alpha_0} (\alpha), \Phi_{\alpha_0}(\tal))+  \dist(\Phi_{\alpha_0} (\tal), \Phi_{\tal_0}(\tal))\\
 &\leq \frac{1}{2}  \dist(\alpha, \tal) + \dist(\Phi_{\alpha_0} (\tal), \Phi_{\tal_0}(\tal))
\end{split}\]
Hence 
 \[\dist(\alpha, \tal) \leq 2  \dist(\Phi_{\alpha_0} (\tal), \Phi_{\tal_0}(\tal))\leq  2 W_2(\alpha_0, \tal_0)  \sup_{t\in [0,T]} e^{-\lambda t} \Lip (X_t^{\tal}, B_{R_T}) \]
 but it is straightforward to deduce from $\bf{(H2)}$ and Gr\"onwall's Lemma, that $\Lip (X_t^{\tal}, B_{R_T})\leq e^{K t}$ and since $\lambda-K=\frac{5 M}{2} \geq 0$ we arrive at 
\begin{equation}\label{lipfloweps}
\dist(\alpha, \tal) \leq 2 W_2(\alpha_0, \tal_0).
\end{equation}
 \end{rem}

 \subsection{Well-posedness  of entropic semi-geostrophic equations}\label{subsec-wellposedness}

 The fact that \pref{sgeps1}-\pref{sgeps2}-\pref{sgeps3} is well-posed then directly follows from Lemma \ref{lembasic} and regularity properties of entropic OT: 
 
 \begin{thm}\label{wpentsg}
 Given $T>0$, $R_0>0$, $\alpha_0$ and $\mu_0$ in $\Prob(B_{R_0})$,  for every $\eps>0$, the system \eqref{sgeps1}-\eqref{sgeps2}-\eqref{sgeps3} admits a unique weak solution on $[0,T]\times \R^d$.
  \end{thm}
 
 \begin{proof}
 Given $\alpha\in \Prob_c(\R^d)$, denote by   $\gamma^\eps$ the solution of $\OT_{\eps}(\alpha, \mu_0)$ (equivalently solve the Schr\"odinger system  \pref{schro1}-\pref{schro2} to obtain the potential $v^\eps$) and set
 \begin{equation}\label{defBeps}
 B^\eps[\alpha](x):=A(x-\int_{\R^d} y \gamma^{\eps}(\dd y\vert x))=A(\nabla v^\eps(x)), \; \forall x\in \R^d.
 \end{equation}
Let us show that the map $\alpha \mapsto B^\eps[\alpha]$ satisfies {\bf{(H1)}}-{\bf{(H2)}}-{\bf{(H3)}}. Concerning {\bf{(H1)}}, the linear growth condition \pref{growthcond} follows from the fact that $\gamma^\eps(.\vert x)$ has support in $B_{R_0}$, so that $\vert B^\eps[\alpha](x) \vert \leq \Vert A \Vert (\vert x\vert+ R_0)$  hence {\bf{(H1)}} holds with the constant $C=\max(1, R_0) \Vert A\Vert$. As for {\bf{(H2)}}, this also follows from the fact that $\gamma^\eps(.\vert x)$ has support in $B_{R_0}$ and the expression of the hessian of $v^\eps$ in \pref{hessve} which yields the bound  (exploding as $\eps\to 0$) 
\[ \Vert D B[\alpha] \Vert \leq \Vert A \Vert \Vert D^2 v^\eps\Vert \leq \Vert A \Vert (1+ 2 \eps^{-1}R_0^2).\] 
Therefore {\bf{(H2)}} holds with the constant   $K_R= \Vert A \Vert (1+ 2 \eps^{-1}R_0^2)$ (depending on $\eps$ but not $R$). Finally, it follows from Proposition \ref{prop:stability} that $B^\eps$ satisfies  {\bf{(H3)}}  with the constant $M_R=\Lambda(\eps, R) \Vert A\Vert$, $\Lambda(\eps, R)$ being the constant in \eqref{displacementsmooth}. Thanks to Lemma \ref{lembasic}, we deduce that there exists a unique solution $\alpha^\eps$ of the system \eqref{sgeps1}-\eqref{sgeps2}-\eqref{sgeps3} with $\alpha_t^\eps$ supported for every $t\in [0,T]$ on the ball of radius $(R_0+1)e^{\max(1, R_0) \Vert A\Vert T}$ but a direct application of Gr\"onwall's Lemma gives that it has to be supported on the ball of radius $2 R_0 e^{\Vert A \Vert T}$ as in our initial definition in section \ref{sec-entreg}. 
\end{proof}
 
 \subsection{On conservation of energy}\label{entropicenergyconserved}
 
 In the physically relevant case $d=3$, $A=J$, one can define the (entropically regularized) total energy
\[E_\eps(\alpha):= \OT_\eps(\alpha, \mu_0) + \int_{\R^3} x_3 \alpha.\]
Of course, the potential energy i.e. the second term is preserved along the flow $\alpha^\eps$ of the entropic semi-geostrophic equation (obtained by Theorem \ref{wpentsg}), by taking a test function depending on $x_3$ only, we readily see that the marginal of $\alpha^\eps$ in the $x_3$-variable is constant in time.  As for the conservation of $t\mapsto \OT_\eps(\alpha_t^\eps, \mu_0)$, one can argue as follows.

\smallskip

 Using a well-known dual formulation (see \cite{Nutz22}, \cite{Leonard12}, \cite{Leonard14}) of entropic optimal transport (or by a direct computation using \eqref{schro1}-\eqref{schro2}), given $\beta \in \Prob_c(\R^3)$, one can also write $\OT_\eps(\beta, \mu_0)$ as the supremum of an unconstrained concave maximization problem
\[\eps+ \sup_{u, v} \left\{\int_{\R^3} u  \dd \mu_0 +\int_{\R^3} v  \dd \beta-\eps \int_{\R^3\times \R^3} e^{\frac{-\vert x-y\vert^2}{2\eps}} e^{\frac{v(x)+u(y)}{\eps}} \beta(\dd x) \mu_0(\dd y)\right\}\]
and this supremum is achieved when $u$ and $v$ are Schr\"odinger potentials between $\beta$ and $\mu_0$. In particular, if we denote by $v^\eps=v^\eps[\beta]$ and $u^\eps=u^\eps[\beta]$ these potentials, when $\eta$ is another compactly supported probability measure, we have
\begin{equation}\label{subgradoteps}
\begin{split}
\OT_\eps(\eta, \mu_0)-\OT_\eps(\beta, \mu_0)& \geq \int_{\R^3} v^\eps (\eta-\beta)\\
&-\eps \int_{\R^3}  \Big(\int_{\R^3} e^{\frac{-\vert x-y\vert^2}{2\eps}} e^{\frac{v^\eps(x)+u^\eps(y)}{\eps}} \mu_0(\dd y)\Big) (\eta-\beta)(\dd x)\\
&= \int_{\R^3} v^\eps (\eta-\beta)
\end{split}
\end{equation}
where passing to the last line, we have used the fact that $v^\eps$ and $u^\eps$ are Schr\"odinger potentials between $\beta$ and $\mu_0$ (so that \eqref{schro1} holds for $\mu=\mu_0$) and that $\eta$ and $\beta$ have same total mass. For $t$ and $s$ in $[0,T]$, applying \eqref{subgradoteps} to the measures $\alpha_s^\eps$, $\alpha_t^\eps$ and the potential $v_t^\eps=v^\eps[\alpha_t^\eps]$, we get, using the fact that $\alpha^\eps$ solves the semi entropic semi-geostrophic equation in the second line, the fact that $J$ is skew-symmetric in the third line, and finally the fact that $\beta \mapsto \nabla v^\eps[\beta]$ satisfies  {\bf{(H3)}} (thanks to  Proposition \ref{prop:stability}) and that $\tau \mapsto \alpha_\tau^\eps$ is Lipschitz for $W_2$ in the last line
\[\begin{split}
\OT_\eps(\alpha_s^\eps, \mu_0)&-\OT_\eps(\alpha_t^\eps, \mu_0)\geq \int_{\R^3} v_t^\eps (\alpha_s^\eps-\alpha_t^\eps)\\
&= \int_t^s  \int_{\R^3} \nabla v_t^\eps(x) J(\nabla v_\tau^\eps(x)) \alpha_\tau^\eps(\dd x) \dd \tau\\
&= \int_t^s  \int_{\R^3} \nabla v_t^\eps(x)  J(\nabla v_\tau^\eps(x)-\nabla v_t^\eps(x)) \alpha_\tau^\eps(\dd x) \dd \tau \\
&\geq - \vert s-t\vert \Vert \nabla v^\eps\Vert_{L^\infty(B_{2 R_0 e^{ T}})} \sup_{\tau \in [t,s]} \Vert \nabla v_\tau^\eps- \nabla v_t^\eps\Vert_{L^{\infty}(B_{2 R_0 e^{ T}})}\\
&\geq -C_\eps (s-t)^2
\end{split}\]
which, reversing the role of $s$ and $t$, yields $\OT_\eps(\alpha_s^\eps, \mu_0)-\OT_\eps(\alpha_t^\eps, \mu_0)=O((s-t)^2)$ hence the preservation of $\OT_\eps(., \mu_0)$ and $E_\eps$ along $\alpha^\eps$.

 \subsection{Discretization}\label{subsec-disc}
 
The fixed-point argument of paragraph \ref{subsec-fixed}  is useful to derive well-posedness for entropic approximations of semi-geostrophic equations as we saw in paragraph \ref{subsec-wellposedness}, it is however useless in practice to design numerical simulations. In this paragraph, we therefore aim to discuss the convergence of an explicit scheme for \eqref{transporteqB} which is in line with \cite{BCM23} and where:
\begin{itemize}
\item the initial condition is approximated by some empirical measure  (which makes pushforward operations explicit)

\item the flow of $B[\alpha]$ is discretized in time by its explicit Euler counterpart (for entropic semi-geostrophic equations and for a discrete $\alpha$, computing $B[\alpha]$ amounts to compute the Schr\"odinger potential between $\alpha$ and $\mu_0$ which can be done efficiently by Sinkhorn's algorithm as in \cite{BCM23}).

\end{itemize}

More precisely, given $\alpha_0 \in \Prob(B_{R_0})$ and $T>0$, we denote by $\alpha$ the solution of \eqref{transporteqB}  obtained as in Lemma \ref{lembasic}. Given $\tal_0 \in \Prob(B_{R_0})$ (which we can think of as a quantized version of $\alpha_0$), $N\in \N^*$ and $\tau=\frac{T}{N}$ a time step, we consider the sequence of measures $\alpha_k^\tau$, $k=0, \ldots, N-1$ defined recursively by
\begin{equation}\label{defalphaktau}
 \alpha_0^\tau=\tal_0, \; \alpha_{k+1}^\tau:=(\id + \tau B[\alpha_k^\tau])_\# \alpha_k^\tau, \; k=0, \ldots, N-1.
\end{equation}
Note that whenever $\tal_0$ is an empirical measure so is $\alpha_k^\tau$ and that $\alpha_k^\tau$ remains supported on $B_{R_T}$ with $R_T=(R_0+1)e^{CT}$ and $C$ is the constant appearing in \pref{growthcond}. Let us finally denote by $\tal^\tau$ the piecewise constant interpolation $t\in [0,T) \mapsto \tal_t^\tau$:
\begin{equation}
\tal^\tau_t= \alpha_k^\tau, \; t\in [k\tau, (k+1)\tau), \; k=0, \ldots, N-1.
\end{equation}
Then we have the following error estimate between the solution $\alpha$ of \eqref{transporteqB} and its discretized counterpart $\tal^\tau$:
\begin{lem}\label{lembasicEuler}
Assuming $B$ satisfies {\bf{(H1)}-\bf{(H2)}-\bf{(H3)}}, there is a positive constant $\Lambda$ depending on $T$, $R_T=(R_0+1)e^{CT}$, $C$, $K=K_{R_T}$ and $M=M_{R_T}$ such that, for $\alpha$ and $\tal^\tau$ as above, one has
\begin{equation}\label{errorestimatetau}
\sup_{t\in [0, T)} W_2(\alpha_t, \tal^\tau_t) \leq \Lambda (\tau + W_2(\alpha_0, \tal_0)).
\end{equation}
\end{lem}

\begin{proof}
For notational simplicity, throughout this proof we just write $R=R_T$.
Setting $L:=\sup_{\beta\in \Prob(B_R)} \Vert B[\beta]\Vert_{L^\infty(B_R)}$, we first obviously have
\begin{equation}\label{ineg00}
 %\max_{k=0, \ldots N-1} W_2(\alpha_{k+1}^\tau, \alpha_k^\tau)\leq L \tau, 
 W_2(\alpha_t, \alpha_s) \leq L \vert t-s\vert, \; (s,t)\in [0,T]^2.
\end{equation}
Hence 
\begin{equation}\label{ineg01}
\sup_{t\in [0, T)} W_2(\alpha_t, \tal^\tau_t) \leq \max _{k=0, \ldots, N} W_2(\alpha_{k \tau}, \alpha_{k}^\tau)+  L\tau.
\end{equation}
For $k=0, \ldots, N$ set $d_k:=W_2(\alpha_{k\tau}, \alpha_{k}^\tau)$. For $k=0, \ldots, N-1$, notice that $\alpha_{(k+1)\tau}={X_k^\tau}_{\#} \alpha_{k\tau}$ where $X_k^\tau:=Y_{(k+1)\tau}$ and 
\[ \frac{d}{d s} Y_s(x)=B[\alpha_s](Y_s(x)), \; Y_{k\tau}(x)=x\]
so that
\[X_{k}^\tau(x)=x+ \tau B[\alpha_{k\tau}](x)+\int_{k\tau}^{(k+1)\tau} (B[\alpha_s](Y_s(x))-B[\alpha_{k\tau}](x)) \mbox{d} s\]
using {\bf{(H2)}-\bf{(H3)}} and the fact that $\vert Y_s(x)-x \vert \leq L (s-k\tau)$ for every $(x,s)\in B_R \times [k\tau, (k+1)\tau]$, we  arrive at
\begin{equation}\label{ineg02}
\Vert X_{k}^\tau-(\id+ \tau B[\alpha_{k\tau}])\Vert_{L^\infty(B_R)} \leq (K+M)L \tau^2.
\end{equation}
By the triangle inequality, we have
\[\begin{split}
d_{k+1}&=W_2({X_k^\tau}_{\#} \alpha_{k\tau}, (\id +\tau B[\alpha_k^\tau])_\# \alpha_k^\tau)\\
&\leq  W_2({X_k^\tau}_{\#} \alpha_{k\tau}, (\id+ \tau B[\alpha_{k\tau}])_\# \alpha_{k\tau}) \\
&+ W_2((\id+ \tau B[\alpha_{k\tau}])_\# \alpha_{k\tau}, (\id+ \tau B[\alpha_{k\tau}])_\# \alpha_{k}^{\tau})\\
&+  W_2( (\id+ \tau B[\alpha_{k\tau}])_\# \alpha_{k}^{\tau}, (\id+ \tau B[\alpha_{k}^{\tau}])_\# \alpha_{k}^{\tau}).
\end{split}\]
Firstly, \eqref{ineg02} yields
\[W_2({X_k^\tau}_{\#} \alpha_{k\tau}, (\id+ \tau B[\alpha_{k\tau}])_\# \alpha_{k\tau}) \leq \Vert X_{k}^\tau-(\id+ \tau B[\alpha_{k\tau}])\Vert_{L^\infty(B_R)} \leq (K+M)L \tau^2.\]
Secondly, assumption {\bf{(H2)}} entails 
\[\begin{split}
W_2((\id+ \tau B[\alpha_{k\tau}])_\# \alpha_{k\tau}, (\id+ \tau B[\alpha_{k\tau}])_\# \alpha_k^\tau)& \leq \Lip(\id + \tau B[\alpha_{k\tau}], B_R) W_2(\alpha_{k\tau}, \alpha_k^\tau)\\
&\leq (1+ K \tau) d_k. 
\end{split}\]
Finally, we deduce from {\bf{(H3)}}
\[W_2( (\id+ \tau B[\alpha_{k\tau}])_\# \alpha_k^\tau, (\id+ \tau B[\alpha_k^\tau])_\# \alpha_k^\tau) \leq  \tau \Vert  B[\alpha_{k\tau}] - B[\alpha_k^\tau] \Vert_{L^\infty(B_R)} \leq M \tau d_k \]
which enables us to conclude that
\[d_{k+1} \leq (K+M)L \tau^2 +(1+(M+K)\tau) d_k\]
which, by the discrete Gr\"onwall's Lemma, yields that for $k=0, \ldots, N$
\[d_k \leq (1+(K+M)\tau)^k (L \tau +d_0) \leq e^{(K+M) T}( L \tau +W_2(\alpha_0, \tal_0))\] 
together with \pref{ineg01}, this proves the desired inequality \pref{errorestimatetau}.
\end{proof}

\begin{rem}
  If we take $\tal_0 = \al_0$ in Lemma \ref{lembasicEuler}, the convergence of the theoretical Euler scheme is a well known fact, even in frameworks with much less regularity (see \cite{CavagnariSavare}). However, the interest of lemma \ref{lembasicEuler} is to give a convergence rate that is linear in $\ta$ which is natural to expect in the case of the first-order Euler scheme with enough regularity. Since \cite{CCL} provides bounds for  higher derivatives of the entropic potentials, the proof could also be adapted to higher order explicit schemes such as Runge-Kutta and provide better convergence rates.
\end{rem}

 \section{Convergence to a weak solution as $\eps \to 0^+$}\label{sec-convergence}
 
 \subsection{Convergence of continuous trajectories}\label{subsec-convcont}
 
 For fixed $T>0$ and $\eps>0$, let $\alpha^\eps \in C([0,T], (\Prob(B_{2R_0e^{\Vert A \Vert T}}), W_2))$ be the unique weak solution of  \pref{sgeps1}-\pref{sgeps2}-\pref{sgeps3} obtained as in Theorem \ref{wpentsg}. To simplify notation, let us set $R:=2R_0e^{\Vert A \Vert T}$.  
 
 \smallskip
 
 Since the velocity $A(\nabla v^\eps)$ can be bounded in $L^\infty(B_{R})$ uniformly in $\eps$, there is a constant $\kappa>0$ such that
 \[ W_2(\alpha_t^\eps, \alpha_s^\eps) \le \kappa \vert t-s\vert, \; \forall \eps>0, \; \forall (s,t)\in [0,T]^2.\]
 Thanks to Arz\`ela-Ascoli Theorem, passing to a vanishing sequence $\eps_n \to 0$ if necessary, one may assume that, for some $\alpha=(\alpha_t)_{t\in [0,T]} \in C([0,T], (\Prob(B_{R}), W_2))$, one has
 \begin{equation}\label{cval}
 \sup_{t\in [0,T]} W_2(\alpha_t^\eps, \alpha_t)\to 0 \mbox{ as $\eps\to 0^+$}.
 \end{equation}
 Our aim is to show that this cluster point $\alpha$ is a weak solution of the un-regularized semi-geostrophic like system \pref{sg3}. By the very definition of $\alpha^\eps$, for $f \in C^1([0,T]\times \R^d)$, one has 
\begin{multline}\label{sgweakeps11}
\int_0^T \int_{\R^d} [\partial_t f + Ax \cdot \nabla f]\alpha_t^\eps(\dd x) \dd t -\int_{\R^d}  f(T,x) \alpha_T^\eps (\dd x)+\int_{\R^d} f(0, x) \alpha_0(\dd x)
\\=\int_0^T \int_{\R^d\times \R^d} A y \cdot \nabla f(t,x) \gamma_t^\eps(\dd x, \dd y) \dd t
\end{multline}
where  for a.e. $t\in [0,T]$, $\gamma_t^\eps$ is the solution of $\OT_\eps(\alpha_t^\eps, \mu_0)$.  Obviously the left-hand side of \pref{sgweakeps11} converges as $\eps\to 0^+$ to
\[\int_0^T \int_{\R^d} [\partial_t f + Ax \cdot \nabla f]\alpha_t(\dd x) \dd t -\int_{\R^d}  f(T,x) \alpha_T (\dd x)+\int_{\R^d} f(0, x) \alpha_0(\dd x)\]
As for the right-hand side, by Banach-Alaoglu's Theorem, one may assume that the family of measures $\gamma_t^\eps(\dd x, \dd y)  \otimes  \dd t$ weakly $*$ converges as $\eps\to 0^+$ to some measure $\theta$ on $B_{R}^2 \times [0,T]$, since   weak $*$ convergence implies weak $*$ convergence of marginals and $\gamma_t^\eps \in \Pi(\alpha_t^\eps, \mu_0)$ for a.e. $t\in [0,T]$, one has
\[\theta(\dd x, \dd y, \dd t)= \gamma_t(\dd x, \dd y)  \otimes  \dd t \mbox{ with $\gamma_t\in \Pi(\alpha_t, \mu_0)$ for a.e. $t\in [0,T]$} \]
 and then 
 \[\begin{split}
\int_0^T \int_{\R^d} [\partial_t f + Ax \cdot \nabla f]\alpha_t(\dd x) \dd t -\int_{\R^d}  f(T,x) \alpha_T (\dd x)+\int_{\R^d} f(0, x) \alpha_0(\dd x)
\\=\int_0^T \int_{\R^d\times \R^d} A y \cdot \nabla f(t,x) \gamma_t (\dd x, \dd y) \dd t.
\end{split}\]
So, to establish that $\alpha$ is a weak solution of  \pref{sg3}, it remains to show that $\gamma_t$ is an optimal transport plan between $\alpha_t$ and $\mu_0$  for a.e. $t\in [0,T]$. Since $\al_t^\eps$ is supported on  $B_{R}$ for all $t \in [0,T]$, by Proposition \ref{prop:uniform},  $\int_{\R^d\times\R^d} |x-y|^2\ga_t^\eps(\dd x, \dd y) - W_2^2(\al_t^\eps,\mu_0)$ converges to $0$ uniformly in $t\in [0,T]$. Thanks to \eqref{cval}, $W_2^2(\al_t^\eps,\mu_0) \to W_2^2(\al_t,\mu_0)$ also uniformly in $t$, so that 
\[\begin{split}
\int_0^T W_2^2(\al_t,\mu_0) \dd t&=\lim_{\eps \to 0} \int_0^T\int_{\R^d\times\R^d} |x-y|^2\ga_t^\eps(\dd x, \dd y) \dd t \\
&= \int_0^T\int_{\R^d\times\R^d} |x-y|^2\gamma_t(\dd x, \dd y) \dd t\end{split} \]
where the second line follows from  the weak $*$ convergence of $\ga^\eps_t (\dd x, \dd y) \otimes \dd t$ to $\gamma_t (\dd x, \dd y) \otimes \dd t$,
and then since $\gamma_t \in \Pi(\alpha_t, \mu_0)$,  for a.e. $t\in [0,T]$, this shows that \pref{sgweak2} holds so that  $\alpha$ is a weak solution of  \pref{sg3}.  In other words, we have shown the following:  \begin{thm}\label{cvepszero}
 Cluster points as $\eps \to 0$ of weak solutions of  \eqref{sgeps1}-\eqref{sgeps2}-\eqref{sgeps3} are weak solutions of \eqref{sg3}.  
 \end{thm}
 
 \begin{rem}
 Note that this  gives yet another proof of existence of weak solutions of semi-geostrophic equations for compactly supported initial data by entropic approximation.  When $d=3$ and $A=J$, combining the considerations of paragraph \ref{entropicenergyconserved} with the  convergence of $\OT_\eps$ to $W_2$, one can readily check that a cluster point $\alpha$ obtained as above, conserves the energy $W_2(\alpha_t, \mu_0)$ (as well as the potential energy of course). This property is actually already well-known for all weak solutions, see \cite{AmbrosioGangbo}. 
 \end{rem}

 \subsection{Convergence of discretizations}

We finally consider the case where we both regularize the semi-geostrophic  equation with an entropy i.e. replace $W^2_2$ by $\OT_\eps$ with small $\eps$ and approximate the solution of \pref{sgeps1}-\pref{sgeps2}-\pref{sgeps3} by using an explicit Euler scheme with time step $\tau$ and quantizing the initial conditon by replacing $\alpha_0$ by some discrete measure $\tal_0$ as in paragraph \ref{subsec-disc}. Combining  Lemma \ref{errorestimatetau} and Theorem \ref{cvepszero}, one easily gets that when letting first $\tau+ W_2(\alpha_0, \tal_0)$ tend to $0$ and then $\eps\to 0$, we obtain (along a subsequence) convergence to a weak solution of  \pref{sg3}. This is of little use in practice to guarantee convergence of the scheme proposed in \cite{BCM23}, so we consider the situation where we let all parameters vanish at the same time in an arbitrary manner, i.e. for $\eps>0$, we consider a time step $\tau_\eps>0$ with $T= N_\eps \tau_\eps$ and a quantized approximation of  the initial $\alpha_0\in \Prob(B_{R_0})$ and of the reference measure $\mu_0\in \Prob(B_{R_0})$
\[\tal_0^\eps:=\frac{1}{M_\eps} \sum_{i=1}^{M_\eps} \delta_{x_i^\eps},\; x_i^\eps \in B_{R_0}; \quad \tmu_0^\eps:=\frac{1}{M_\eps} \sum_{i=1}^{M_\eps} \delta_{y_i^\eps},\; y_i^\eps \in B_{R_0},\]
 and assume that
 \begin{equation}\label{wpr}
 \tau_\eps + W_2(\alpha_0, \tal_0^\eps) + W_2(\mu_0, \tmu_0^\eps)\to 0, \mbox{ as } \eps \to 0^+.
 \end{equation}
 We then construct a piecewise constant curve of measures $t\in [0, T] \mapsto \tal_t^\eps$ by an explicit Euler scheme as in \ref{subsec-disc}, i.e.:
\[\tal_t^\eps=\alpha_k^\eps, \; t\in [k \tau_\eps, (k+1) \tau_\eps), \; k=0, \ldots, N_\eps-1\]
with
\[\alpha_0^\eps=\tal_0^\eps, \; \alpha_{k+1}^\eps =(\id+ \tau_\eps B^\eps[\alpha_k^\eps])_\# \alpha_k^\eps, \; k= 0,\ldots, N_\eps-1\]
 where $B^\eps$ is defined  through the solution of $\OT_\eps(\alpha, \tmu_0^\eps)$ as in \eqref{defBeps}. Observing that $W_2(\tal_t^\eps, \tal_s^\eps) \leq  \kappa (\vert t-s\vert+\tau)$ for every $t,s$ in $[0,T]$ and a constant $\kappa$ independent of $\eps$, passing along a (not relabeled) vanishing sequence $\eps_n\to 0$, we may assume that for some $\alpha=(\alpha_t)_{t\in [0,T]} \in C([0,T], (\Prob(B_{R}), W_2))$ (where we recall that we have set $R:=2 R_0 e^{\Vert A\Vert T}$) one has
 \begin{equation}\label{cvtaleps}
 \sup_{t\in [0,T)} W_2(\tal_t^\eps, \alpha_t)\to 0 \mbox{ as $\eps\to 0^+$},
 \end{equation}
 and  $W_2(\alpha_t, \alpha_s) \leq \kappa \vert t-s\vert$, for all $s$ and $t$ in $[0,T]$. We then have the following convergence result, which shows that cluster points of the previous approximations are weak solutions of \pref{sg3}:

 \begin{thm}\label{cvtepszero}
 If $\alpha$ is obtained as a cluster point of the discretized entropic regularization $(\tal_t^\eps)_{t\in [0,T)}$ i.e. such that  \eqref{cvtaleps} holds, then $\alpha$ is a weak solution of \eqref{sg3}.  
 \end{thm}

 \begin{proof}
 Let $f\in C^1([0,T]\times \R^d)$,  observe that 
 \begin{equation}\label{ff1}
 \begin{split}
 \int_0^T \int_{\R^d} \partial_t f \tal_t^\eps&= \sum_{k=0}^{N_\eps-1}   \int_{\R^d}  \int_{k\tau_\eps }^{(k+1)\tau_\eps}  \partial_t f \alpha_k^\eps= \sum_{k=0}^{N_\eps-1} \int_{\R^d} (f((k+1)\tau_\eps,\cdot)-f(k\tau_\eps, \cdot)) \alpha_k^\eps  \\
& =\sum_{k=1}^{N_\eps-1} \int_{\R^d} f(k\tau_\eps,\cdot) (\alpha_{k-1}^\eps-\alpha_k^\eps)+ \int_{\R^d} f(T,\cdot) \alpha_{N-1}^\eps -\int_{\R^d} f(0,\cdot) \alpha_{0}^\eps
 \end{split}
 \end{equation}
 note that the last two terms in \eqref{ff1} converge  as $\eps\to 0$ respectively to  $\int_{\R^d} f(T, \cdot) \alpha_T$ and $-\int_{\R^d} f(0, \cdot) \alpha_{0}$. Setting $B_k^\eps=B^\eps[\alpha_k^\eps]$ and denoting by $\gamma^\eps_{k-1}$ the solution of $\OT_\eps(\alpha_{k-1}^\eps, \tmu^\eps_0)$ and using the fact that $\alpha_k^\eps=(\id+\tau_\eps B_k^\eps)_\# \alpha_{k-1}^\eps$, enables to rewrite
 \[\begin{split} \int_{\R^d} f(k\tau_\eps,\cdot) (\alpha_{k-1}^\eps-\alpha_k^\eps)&= \int_{\R^d} (f(k\tau_\eps,x)-f(k \tau_\eps, x+ \tau_\eps B_{k-1}^\eps(x)) ) \alpha_{k-1}^\eps(\dd x)\\
 &=-\tau_\eps \int_{\R^d} \nabla f(k\tau_\eps,x) \cdot B_{k-1}^\eps(x) \alpha_{k-1}^\eps(\dd x)+ o(\tau_\eps)\\
 &=\int_{ (k-1)\tau_\eps}^{k\tau_\eps} \int_{\R^d\times \R^d} \nabla f(t,x) \cdot A(y-x) \gamma_{k-1}^\eps(\dd x, \dd y) \dd t+ o(\tau_\eps).
 \end{split}\]
 Considering the piecewise constant curve of plans $t\mapsto \tgam^\eps_t$ defined by $\tgam^\eps_t=\gamma_{k}^\eps$ for $t\in [k\tau_\eps, (k+1)\tau_\eps)$, recalling that $N_\eps \tau_\eps=T$, we thus have
 \begin{equation}\label{ff2}
 \begin{split}
 \int_0^T \int_{\R^d} \partial_t f \tal_t^\eps=\int_0^T \int_{\R^d\times \R^d} \nabla f(t,x) \cdot A(y-x) \tgam^\eps_t(\dd x, \dd y) \dd t\\
+\int_{\R^d} f(T,x) \alpha_T(\dd x)-\int_{\R^d} f(0,x) \alpha_{0}(\dd x) +o(1).
 \end{split}
 \end{equation}
 As in paragraph \ref{subsec-convcont}, we may assume (possibly after an extraction) that $\tgam_t^\eps(\dd x, \dd y)  \otimes  \dd t$ weakly $*$ converge as $\eps\to 0^+$ to some measure  of the form $\gamma_t(\dd x, \dd y)  \otimes  \dd t$. Then, exactly as in paragraph \ref{subsec-convcont}, Proposition \ref{prop:uniform} ensures that for a.e. $t$, $\gamma_t \in \Pi(\alpha_t, \mu_0)$ is an optimal transport plan between $\alpha_t$ and $\mu_0$ i.e. satisfies \eqref{sgweak2}. Letting $\eps\to 0^+$ in \eqref{ff2}, thanks to \eqref{cvtaleps} and $\gamma_t \in \Pi(\alpha_t, \mu_0)$, we get
 \[\begin{split}
 \int_0^T \int_{\R^d} [\partial_t f   + A x\cdot \nabla f]  \alpha_t=\int_0^T \int_{\R^d\times \R^d} \nabla f(t,x) \cdot A y \; \gamma_t(\dd x, \dd y) \dd t\\
+\int_{\R^d} f(T,x) \alpha_T(\dd x)-\int_{\R^d} f(0,x) \alpha_{0}(\dd x)
 \end{split}\]
and since $\gamma_t$ is an optimal transport plan between $\alpha_t$ and $\mu_0$, $\alpha$ is a weak solution of \eqref{sg3}.
 \end{proof}

\bigskip
  
{\bf Acknowledgments:}   G.C. acknowledges the support of the Lagrange Mathematics and Computing Research Center. The authors wish to thank Jean-David Benamou and David Bourne for helpful comments and suggestions.

\bibliographystyle{plain}

\bibliography{bibli}

\end{document}